\documentclass{amsart}
\usepackage{graphicx}
\usepackage{mathrsfs}
\usepackage{enumitem,bbm}

\newtheorem{proposition}{Proposition}
\numberwithin{proposition}{section}
\newtheorem{lemma}{Lemma}
\numberwithin{lemma}{section}
\newtheorem{theorem}{Theorem}
\numberwithin{theorem}{section}

\newcounter{picno}
\newcommand{\pic}[1]{\refstepcounter{picno}\label{#1}}

\newcommand{\la}{\lambda}

\DeclareMathOperator{\charac}{char}

\title{Brion's theorem for Gelfand-Tsetlin polytopes}
\thanks{This work was supported in part by the Simons Foundation and the Moebius Contest Foundation for Young Scientists.}

\author{I. Makhlin}
\address{Igor Makhlin:\newline
National Research University Higher School of Economics, 
International Laboratory of Representation Theory and Mathematical Physics\\
Vavilova str. 7, 117312, Moscow, Russia,\newline
{\it and }\newline
Landau Institute for Theoretical Physics, prospekt Akademika Semenova 1A, 142432 Chernogolovka, Russia}
\email{imakhlin@mail.ru}

\begin{document}

\maketitle

\begin{abstract}
This work is motivated by the observation that the character of an irreducible $\mathfrak{gl}_n$-module (a Schur polynomial), being the sum of exponentials of integer points in a Gelfand-Tsetlin polytope, may be expressed via Brion's theorem. The main result is that in the case of a regular highest weight the contributions of all non-simplicial vertices vanish while the number o simplicial vertices is $n!$ and their contributions are precisely the summands in Weyl's character formula.
\end{abstract}

\section{Introduction}

In this section we recall some preliminaries and then state the main result.

\subsection{Representations of $\boldsymbol{\mathfrak{gl}_n(\mathbb C)}$ and Gelfand-Tsetlin polytopes}
Consider the Lie algebra $\mathfrak{gl}_n(\mathbb{C})$ comprised of complex $(n\times n)$-matrices. Within the Cartan subalgebra of diagonal matrices we consider the basis of matrices with a single nonzero entry equal to one. The dual basis constitutes a basis in the lattice of integral weights; this will be our basis of choice therein.

Integral dominant weights are then precisely those the coordinates of which comprise a non-increasing sequence of integers. Every such weight $\lambda$ provides a finite-dimensional irreducible representation~$L_\lambda$.

We start out by writing Weyl's formula for the character of such a representation: 
$$
\charac{L_\lambda}=\sum_{w\in S_n} w\bigg(\frac{e^\lambda}{\prod_{1\le i<j\le n}(1-x_j/x_i)}\bigg).
$$
Here for a weight $\mu$ we set $e^\mu=x_1^{\mu_1}\ldots x_n^{\mu_n}$ and make use of the fact that the Weyl group of our algebra is $S_n$ acting on weights by permuting the coordinates.

The number $n$ and the weight $\lambda=(\la_1,\dots,\la_n)$ are fixed throughout the paper.

The Gelfand-Tsetlin basis in $L_\lambda$ is parametrized by the so-called Gelfand-Tsetlin (GT) patterns. Each such pattern is a number triangle $\{A_{i,j}\}$ with $0\le i\le n-1$ and $1\le j\le n-i$. The top row of the triangle is simply $A_{0,j}=\lambda_j$. The remaining elements are arbitrary integers satisfying 
\begin{equation}\label{gt}
A_{i,j}\ge A_{i+1,j}\ge A_{i,j+1}.
\end{equation}
These patterns are commonly visualized in the following manner:
\begin{center}
\begin{tabular}{ccccccc}
$A_{0,1}$ &&$ A_{0,2}$&& $\ldots$ && $A_{0,n}$\\
&$A_{1,1}$ &&$ \ldots$&& $A_{1,n-1}$ &\\
&&$\ldots$ &&$ \ldots$& &\\
&&&$A_{n-1,1}$ &&&
\end{tabular}
\end{center}
In other words, every element is no greater than its upper-left neighbor and no less than its upper-right neighbor except for, naturally, the elements in row 0 (row $i$ is comprised of elements of the form $A_{i,\cdot}$).

Denote the set of GT-patterns via $\mathbf{GT}_\lambda$. Let $A\in \mathbf{GT}_\lambda$ correspond to the basis vector $v_A$ with weight $\mu_A$. With respect to the chosen basis the weight $\mu_A$ has coordinates
\begin{equation}\label{weight}
(\mu_A)_i=\sum_j A_{i-1,j}-\sum_j A_{i,j}
\end{equation}
where $A_{n,j}=0$.

Herefrom one obtains the well-known combinatorial formula for irreducible $\mathfrak{gl}_n$-characters, i.e. Schur polynomials $s_\lambda$:
\begin{equation}\label{schurcomb}
s_\lambda(x_1,\dots,x_n)=\sum_{A\in\mathbf{GT}_\lambda}e^{\mu_A}.
\end{equation}
(It should be noted that here we understand Schur polynomials in a generalized sense: the coordinates of $\lambda$ may be negative which makes $s_\lambda$ a Laurent polynomial and not necessarily a common polynomial.)

Now observe that every GT-pattern may be viewed as an integer point in a real space of dimension ${n+1}\choose 2$. One sees that these patterns then comprise precisely the set of integer points in a certain polytope, the Gelfand-Tsetlin polytope.

By definition, the Gelfand-Tsetlin polytope $GT_\lambda$ is a polytope in the space $\mathbb{R}^{{n+1}\choose 2}$ with coordinates labeled by pairs of integers $0\le i\le n-1$ and $1\le j\le n-i$ and is defined by the conditions $A_{0,j}=\lambda_j$ and the inequalities~(\ref{gt}). It is not hard to see that it is a bounded polytope of codimension no less than $n$.

GT-patterns then indeed comprise the set of that polytope's integer points and formula~(\ref{schurcomb}) tells us that the Schur polynomial is a sum of certain exponentials of these integer points. Such sums of exponentials may be computed via Brion's theorem which we now discuss.

\subsection{Brion's theorem}
Consider the vector space $\mathbb{R}^m$ with a fixed basis and lattice of integer points $\mathbb{Z}^m\subset\mathbb{R}^m$. For any subset $P\subset\mathbb{R}^m$ we define its generating function 
$$
S(P)=\sum_{a\in P\cap\mathbb{Z}^m} \mathbf x^a,
$$
a formal Laurent series in the variables $x_1,\dots,x_m$. (The formal exponential of point $a$ is defined as $\mathbf x^a=x_1^{a_1}\cdots x_m^{a_m}$.)

Let $P$ be a convex rational polytope (the intersection of a finite set of half-spaces given by non-strict linear inequalities with integer coefficients). In this case there exists a polynomial $q\in\mathbb{Z}[x_1^{\pm 1}\!\!,\dots,x_m^{\pm 1}]$ such that the product $qS(P)$ is finite, i.e. is also a (Laurent) polynomial. Moreover, the rational function $\tfrac{qS(P)}q$ is independent of the choice of $q$ and will be denoted $\sigma(P)$ (it is sometimes referred to as the {\it integer-point transform}).

At every vertex $v$ of $P$ we have the tangent cone $C_v$. Brion's theorem is the following identity in the field of rational functions.

\begin{theorem}[\cite{bri}, \cite{kp}]\label{brion} One has the identity
$$
\sigma(P)=\sum_{v\text{ vertex of }P}\sigma(C_v).
$$
\end{theorem}

The books~\cite[ch. 13]{barv} and~\cite[ch. 9]{beckrob} contain outstanding discussions of these topics.

\subsection{The main result}
The formal exponential of a point in the ${n+1}\choose 2$-dimensional space containing the GT-polytope is a monomial in ${n+1}\choose 2$ variables. Denote these variables $t_{i,j}$ labeling them in accordance with the coordinates.

Formula~(\ref{weight}) shows that for a GT-pattern $A$ the monomial $e^{\mu_A}$ is obtained from the monomial $\mathbf t^A$ (the formal exponential of a point) via the specialization
$$
t_{i,j}\mapsto
\begin{cases}
x_1&\text{when }\,i=0,\\
x_{i}^{-1} x_{i+1}&\text{when }\,i>0.
\end{cases}
$$
In general, for a rational function $Q$ in the variables $t_{i,j}$ we will write $F(Q)$ for the expression obtained from $Q$ by applying the above specialization.

Herefrom via fromula~(\ref{schurcomb}) we deduce the following fact.

\begin{theorem} We have
$$
s_\lambda(x_1,\dots,x_n)=F(S(GT_\lambda)).
$$
\end{theorem}

The right-hand side above may be computed via Brion's theorem. The goal of this paper is to show how the resulting expression turns out to be the classical alternating formula for Schur polynomials (Weyl's character formula).

The answer is particularly instructive in the case of a regular weight $\lambda$, i.e. the $\lambda_j$ being pairwise distinct. It is provided by the below theorem which can be viewed as our main result.

\begin{theorem}\label{main}
For a regular dominant integral weight $\la$ the polytope $GT_\lambda$ has exactly $n!$ simplicial vertices. For every permutation $w\in S_n$ there is exactly one simplicial vertex $v$ with $F(\mathbf t^v)=w(e^\lambda)$. This vertex $v$ satisfies
$$
F(\sigma(C_v))=w\bigg(\frac{e^\lambda}{\prod_{i<j}(1-x_j/x_i)}\bigg)
$$
($C_v$ is the tangent cone). For every non-simplicial vertex one has $F(\sigma(C_v))=0$.
\end{theorem}

In other words, the contributions of simplicial vertices constitute precisely the $n!$ summands in Weyl's character formula while the contributions of all other vertices vanish.

The case of a singular $\lambda$ will be considered in the end of the paper. The answer in that case is not quite as neat, nonetheless, one has a similar theorem (Theorem~\ref{sing}) describing the connection between Brion's formula and Weyl's character formula.

\section{Proof of the main theorem}

\vglue-2pt

We have fixed a regular dominant integral weight $\la$. To prove theorem~\ref{main} we classify the vertices of $GT_\lambda$ by matching them with certain graphs. It turns out that that simplicial vertices are precisely those corresponding to acyclic graphs. After that, with the help of an inductive argument, we show that the contributions of vertices corresponding to graphs with cycles vanish. Finally, we compute the contributions of the remaining (i.e. simplicial) vertices to complete the proof. 

According to its definition, our polytope is the intersection of the subspace given by the conditions $A_{0,j}=\lambda_j$ and of the half-spaces given by the inequalities~(\ref{gt}). We claim that in view of $\la$ being regular each of the latter inequalities provides a facet of $GT_\lambda$ (the bounding hyperplane of the half-space intersects the polytope in a facet). Indeed, otherwise one of the inequalities would be a consequence of the others which is obviously not the case.

Next, vertices of the polytope can be characterized with the following proposition.

\begin{proposition}\label{vert}
A GT-pattern $A$ constitutes a vertex of $GT_\lambda$ if and only if for any $1\le i\le n-1$ and $1\le j\le n-i$ the element $A_{i,j}$ is equal to at least one of its upper neighbors $A_{i-1,j}$ and $A_{i-1,j+1}$.
\end{proposition}

\begin{proof}
This follows directly from the fact that a point is a vertex if and only if the set of facets it is contained in is maximal by inclusion (among the polytope's points).
\end{proof}

The following notion will turn out to be very handy. With every vertex $v$ of $GT_\la$ we associate a graph $\Gamma_v$.

First, consider the graph $T$ with $n(n+1)/2$ nodes corresponding to the elements of a GT-pattern (the nodes also denoted by pairs $(i,j)$) and such that for any $1\le i\le n-1$, $1\le j\le n-i$ the node $(i,j)$ is adjacent to $(i-1,j)$ and $(i-1,j+1)$.

Now $\Gamma_v$ can be defined as a subgraph of $T$ containing all of its nodes. An edge of $T$ lies in $\Gamma_v$ if and only if the two corresponding elements of GT-pattern $v$ are equal. Here are two examples of such graphs.

\begin{center}
\setlength{\unitlength}{1mm}
\begin{picture}(95,40)
\put(5,35){\circle{4}}
\put(4.2,34){5}
\put(17,35){\circle{4}}
\put(16,34){4}
\put(29,35){\circle{4}}
\put(28.2,34){2}
\put(41,35){\circle{4}}
\put(40.2,34){0}
\put(11,25){\circle{4}}
\put(10,24){4}
\put(23,25){\circle{4}}
\put(22,24){4}
\put(35,25){\circle{4}}
\put(34.2,24){0}
\put(17,15){\circle{4}}
\put(16,14){4}
\put(29,15){\circle{4}}
\put(28.2,14){0}
\put(23,5){\circle{4}}
\put(22,4){4}
\put(15.97,33.29){\line(-3,-5){3.94}}
\put(18.03,33.29){\line(3,-5){3.94}}
\put(12.03,23.29){\line(3,-5){3.94}}
\put(21.97,23.29){\line(-3,-5){3.94}}
\put(18.03,13.29){\line(3,-5){3.94}}
\put(39.97,33.29){\line(-3,-5){3.94}}
\put(33.97,23.29){\line(-3,-5){3.94}}
\pic{ex1}

\put(61,35){\circle{4}}
\put(60.2,34){5}
\put(73,35){\circle{4}}
\put(72,34){4}
\put(85,35){\circle{4}}
\put(84.2,34){2}
\put(97,35){\circle{4}}
\put(96.2,34){0}
\put(67,25){\circle{4}}
\put(66.2,24){5}
\put(79,25){\circle{4}}
\put(78,24){4}
\put(91,25){\circle{4}}
\put(90.2,24){0}
\put(73,15){\circle{4}}
\put(72,14){4}
\put(85,15){\circle{4}}
\put(84.2,14){0}
\put(79,5){\circle{4}}
\put(78,4){4}
\put(62.03,33.29){\line(3,-5){3.94}}
\put(74.03,33.29){\line(3,-5){3.94}}
\put(77.97,23.29){\line(-3,-5){3.94}}
\put(74.03,13.29){\line(3,-5){3.94}}
\put(95.97,33.29){\line(-3,-5){3.94}}
\put(89.97,23.29){\line(-3,-5){3.94}}
\pic{ex2}
\end{picture}\kern20pt\\
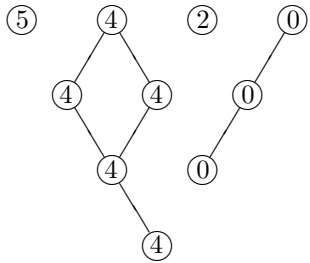
\vglue-7pt\centerline{\small \textbf{Fig. 1}\kern130pt \textbf{Fig. 2}}
\vskip10pt
\end{center}

Thus the edges of $\Gamma_v$ correspond to facets containing $v$ or, in other words, to facets of the tangent cone $C_v$. (It appears that the vertices of GT-polytopes were first described in terms of such graphs in the paper~\cite{kst}.)

Consider the connected components of such a graph $\Gamma_v$. In view of proposition~\ref{vert} there are exactly $n$ of them and each component contains exactly one vertex from row 0. Moreover, any component $\Delta$ has the following property:
\begin{enumerate}[label=\Alph*.]
\item If $\Delta$ contains both nodes $(i,j)$ and $(i,j+1)$, then $\Delta$ also contains $(i-1,j+1)$ and $(i+1,j)$, i.e. the two common neighbors of the former two nodes in $T$.
\end{enumerate}

We will refer to a subgraph $\Delta\subset T$ as \textit{ordinary} if it is connected, induced and has property A. In particular, we see that both the top and bottom rows containing nodes from $\Delta$ necessarily contain exactly one node therefrom. Every connected component of a graph $\Gamma_v$ is an ordinary subgraph.

With every ordinary subgraph $\Delta$ we associate a cone $C_\Delta$ in the ${n+1}\choose 2$-dimensional space containing $GT_\lambda$. A point $x$ belongs to $C_\Delta$ if and only if it has the following three properties.
\begin{enumerate}
\item For every $(i,j)$ not a node of $\Delta$ we have $x_{i,j}\,{=}\,0$.

\item Let $(i_\Delta,j_\Delta)$ be the single top node in $\Delta$. Then $x_{i_\Delta,j_\Delta}=0$.

\item For every edge in $\Delta$ the two coordinates of $x$ corresponding to the edge's endpoints satisfy the according inequality of form~(\ref{gt}).
\end{enumerate}
One sees that $C_\Delta$ is indeed a cone with its vertex at the origin.

We introduced this notion for the following reason. Consider a vertex $v$ of $GT_\lambda$. Then
$$
C_v=v+\sum_{\Delta\text{ component of }\Gamma_v} C_\Delta
$$
(Minkowski sum). The above sum is ``direct'' in the sense that the sum of the linear hulls of cones $C_\Delta$ is direct. The following key identity ensues:
\begin{equation}\label{decomp}
F(\sigma(C_v))=F(\mathbf t^v)\prod_{\Delta\text{ component of }\Gamma_v}F(\sigma(C_\Delta)).
\end{equation}

Now we distinguish the set of simplicial vertices.

\begin{proposition}
A vertex $v$ is simplicial if and only if the graph $\Gamma_v$ is acyclic.
\end{proposition}

(One sees  that Fig.~1 on page~\pageref{ex1} provides an example of a non-simplicial vertex while Fig.~2 provides an example of a simplicial one.)

\begin{proof}
A vertex $v$ is simplicial if and only if the number of facets containing it is equal to $\dim{GT_\lambda}={{n+1}\choose 2}-n$. The number of facets containing $v$ is equal to the number of edges in the graph $\Gamma_v$. However, if at least one of the connected components in $\Gamma_v$ is not a tree, then the component contains no fewer edges than nodes and the total number of edges is, therefore, greater than ${{n+1}\choose 2}-n$.
\end{proof}

Note that the acyclicity of $\Gamma_v$ is equivalent to no component therein containing two nodes from the same row.

In view of~(\ref{decomp}), the last part of Theorem~\ref{main} will now follow from the below fact.

\begin{theorem}\label{zero}
If an ordinary subgraph $\Delta$ contains a cycle, then $F(\sigma(C_\Delta))\,{=}\,0$.
\end{theorem}

To prove this theorem we will require an explicit description of the edges of $C_\Delta$.

\begin{proposition}\label{edges}
Let $\varepsilon$ be the minimal integer direction vector of an edge of  $C_\Delta$. Then there exists a pair of nonintersecting ordinary subgraphs $\Delta_1$, $\Delta_2$ such that every node of $\Delta$ is a node of one of them and with the following property: if $\Delta_1$ contains the node $(i_\Delta,j_\Delta)$, then all the coordinates of $\varepsilon$ belonging to $\Delta_1$ are zero while the coordinates belonging to $\Delta_2$ are all the same and equal to either $1$ or $-1$ (depending on the mutual arrangement of $\Delta_1$ and $\Delta_2$).
\end{proposition}

Here are several examples of such vectors $\varepsilon$ for various $\Delta$. Solid lines denote the edges of subgraphs $\Delta_1$ and $\Delta_2$ while dotted lines denote the remaining edges of $\Delta$. We will make use of these examples later on.

\begin{center}
\setlength{\unitlength}{1mm}
\begin{picture}(95,50)
\put(23,45){\circle{4}}
\put(22.2,44){0}
\put(17,35){\circle{4}}
\put(16.2,34){1}
\put(11,25){\circle{4}}
\put(10.2,24){1}
\put(23,25){\circle{4}}
\put(22.2,24){1}
\put(17,15){\circle{4}}
\put(16.2,14){1}
\multiput(21.97,43.29)(-0.394,-0.657){11}{\circle*{0}}
\put(15.97,33.29){\line(-3,-5){3.94}}
\put(18.03,33.29){\line(3,-5){3.94}}
\put(12.03,23.29){\line(3,-5){3.94}}
\put(21.97,23.29){\line(-3,-5){3.94}}
\pic{ex3}

\put(36,45){\circle{4}}
\put(35.2,44){0}
\put(30,35){\circle{4}}
\put(29.2,34){0}
\put(42,35){\circle{4}}
\put(40.5,34){-1}
\put(36,25){\circle{4}}
\put(34.5,24){-1}
\put(48,25){\circle{4}}
\put(46.5,24){-1}
\put(42,15){\circle{4}}
\put(40.5,14){-1}
\put(34.97,43.29){\line(-3,-5){3.94}}
\multiput(37.03,43.29)(0.394,-0.657){11}{\circle*{0}}
\multiput(31.03,33.29)(0.394,-0.657){11}{\circle*{0}}
\put(40.97,33.29){\line(-3,-5){3.94}}
\put(43.03,33.29){\line(3,-5){3.94}}
\put(37.03,23.29){\line(3,-5){3.94}}
\put(46.97,23.29){\line(-3,-5){3.94}}
\pic{ex4}

\put(61,45){\circle{4}}
\put(60.2,44){0}
\put(55,35){\circle{4}}
\put(54.2,34){0}
\put(67,35){\circle{4}}
\put(65.5,34){-1}
\put(61,25){\circle{4}}
\put(59.5,24){-1}
\put(67,15){\circle{4}}
\put(65.5,14){-1}
\put(59.97,43.29){\line(-3,-5){3.94}}
\multiput(62.03,43.29)(0.394,-0.657){11}{\circle*{0}}
\multiput(56.03,33.29)(0.394,-0.657){11}{\circle*{0}}
\put(65.97,33.29){\line(-3,-5){3.94}}
\put(62.03,23.29){\line(3,-5){3.94}}
\pic{ex5}

\put(86,45){\circle{4}}
\put(85.2,44){0}
\put(80,35){\circle{4}}
\put(79.2,34){0}
\put(92,35){\circle{4}}
\put(90.5,34){-1}
\put(86,25){\circle{4}}
\put(85.2,24){0}
\put(92,15){\circle{4}}
\put(91.2,14){0}
\put(84.97,43.29){\line(-3,-5){3.94}}
\multiput(87.03,43.29)(0.394,-0.657){11}{\circle*{0}}
\put(81.03,33.29){\line(3,-5){3.94}}
\multiput(90.97,33.29)(-0.394,-0.657){11}{\circle*{0}}
\put(87.03,23.29){\line(3,-5){3.94}}
\pic{ex6}
\end{picture}\kern30pt\\
\vglue-30pt
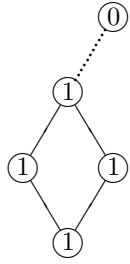
\centerline{\small \textbf{Fig. 3}\kern38pt \textbf{Fig. 4}\kern38pt \textbf{Fig. 5}\kern38pt \textbf{Fig. 6}}
\vskip10pt
\end{center}

\begin{proof}{Proof of Proposition \ref{edges}}
Suppose that $\Delta$ contains at least two nodes, i.e. $C_\Delta$ is not a point. We will employ the fact that the edges of a cone are rays with endpoints in the vertex for which the set of facets containing them is maximal by inclusion among such rays.

First of all, if the coordinates of $\varepsilon$ take at least three distinct values, then, clearly, we may choose one of those values and set all of the corresponding coordinates equal to another of the occurring values in such a way that the set of facets containing the edge increases.

Next, it is also clear that the coordinates taking a specific value comprise the set of nodes of an ordinary subgraph. Moreover, $\varepsilon_{i_\Delta,j_\Delta}=\nobreak0$ by definition of $C_\Delta$ and the minimality and integrality of $\varepsilon$ shows that all the nonzero coordinates of $\varepsilon$ are equal to $\pm1$.
\end{proof}

\begin{proof}{Proof of Theorem~\ref{zero}}
We will proceed by induction on the number of nodes in $\Delta$. We discuss the step of our induction by considering three cases, the base being covered by Case 3.

{\it Case 1.} Graph $\Delta$ contains just one node in row $i_\Delta+1$. We denote that node $(i_\Delta+1,j_1)$.

From Proposition~\ref{edges} it is evident that $C_\Delta$ has just one edge $e$ lying outside of the space $\{x_{i_\Delta+1,j_1}=0\}$, all coordinates of its direction vector $\varepsilon$ other than $\varepsilon_{i_\Delta,j_\Delta}$ are equal to $\pm 1$ (cf. Fig.~3).

Let $\Delta'$ be the ordinary subgraph obtained from $\Delta$ by deletion of the node $(i_\Delta,j_\Delta)$. We see that
$$
C_\Delta=C_{\Delta'}+e
$$
and
$$
F(\sigma(C_\Delta))=F(\sigma(C_{\Delta'})/(1-\mathbf t^\varepsilon))=0
$$
via the induction hypothesis. ($\Delta'$ contains a cycle.)

{\it Case 2.} The graph $\Delta$ contains two nodes in row $i_\Delta+1$ and, on top of that, for some $i_1>i_\Delta$ row $i_1+1$ contains more nodes than~$i_1$.

To compute $\sigma(C_\Delta)$ we define a certain rational polyhedron $D$ the set of integer points in which coincides with the set of integer points in $C_\Delta$. First we define the cone $C'=C_\Delta+\delta$ where $\delta$ is a rational vector with three properties.
\begin{enumerate}
\item $-\delta\in C_\Delta$, i.e. $C_\Delta\subset C'$.
\item $\delta$ is small enough, i.e. the set of integer points in $C'$ coincides with the set of integer points in $C_\Delta$.
\item $\delta_{i_\Delta+1,j_\Delta}>0$.
\end{enumerate}
Now we set $D=C'\cap\{x_{i_\Delta+1,j_\Delta}\le 0\}$.

Let us compute $F(\sigma(D))=F(\sigma(C_\Delta))$. The vertices of $D$ are intersections of the space $\{x_{i_\Delta+1,j_\Delta}=0\}$ with edges of the cone $C'$. Every edge of $C'$ intersecting that space is a translation of an edge of $C_\Delta$ with direction vector $\varepsilon$ such that $\varepsilon_{i_\Delta+1,j_\Delta}=-1$. Choose such an $\varepsilon$ and denote the corresponding vertex of $D$ via $v'$. We show that $F(\sigma(C_{v'}))=0$ ($C_{v'}$ being the tangent cone to $D$) which, in view of Brion's theorem, implies the desired identity. 

Indeed, $C_{v'}$ is equal to the direct sum $v'+C_{\Delta_1}+C_{\Delta_2}$ where $\Delta_1$ and $\Delta_2$ are defined as follows: $\Delta_1$ is the ordinary subgraph comprised of nodes $(i,j)$ with $\varepsilon_{i,j}=0$ and $\Delta_2$ is the ordinary subgraph comprised of nodes $(i,j)$ with $\varepsilon_{i,j}=-1$ and the node $(i_\Delta,j_\Delta)$. Diminishing $\delta$ if necessary, we obtain $S_{C_{v'}}=C_{\Delta_1}+C_{\Delta_2}$.

It follows that
$$
F(\sigma(C_{v'}))=F(\sigma(C_{\Delta_1})\sigma(C_{\Delta_2}))=0,
$$
since at least one of $\Delta_1$ and $\Delta_2$ contains at least two nodes from row $i_1+1$ and, consequently, contains a cycle.

{\it Case 3.} The graph $\Delta$ contains two nodes in row $i_\Delta+1$ and we are not within Case 2. This means that for a certain $i_1$ the graph $\Delta$ contains two nodes in each of rows $i_\Delta+1,\dots,i_1$ and one vertex in any row with number greater than $i_1$ but containing nodes from $\Delta$.

In this case our argument repeats that in the previous case except for the last step. $D$ does contain vertices $v'$ for which both $\Delta_1$ and $\Delta_2$ are acyclic. However, one sees that there are exactly two such vertices: $v'(1)$ and $v'(2)$ corresponding to direction vectors  $\varepsilon(1)$ and $\varepsilon(2)$. Denote the corresponding subgraphs via $\Delta_1(1)$, $\Delta_2(1)$, $\Delta_1(2)$, $\Delta_2(2)$. Then the left node of $\Delta$ in row  $i\in[i_\Delta+1,i_1]$ belongs to $\Delta_1(i)$ while the right one belongs to $\Delta_2(i)$ ($i=1,2$). The difference is that all nodes of $\Delta$ lying in rows below $i_1$ belong to $\Delta_2(1)$ but they also belong to $\Delta_1(2)$. (Figs.~5 and~6 provide examples of such vectors $\varepsilon(1)$ and $\varepsilon(2)$. Fig.~4 provides an example of a vector $\varepsilon$ for which $\Delta_2$ does contain a cycle and the induction hypothesis is applicable.)

A direct computation shows that
$$
F(\sigma(C_{\Delta_1(1)})\sigma(C_{\Delta_2(1)}))+F(\sigma(C_{\Delta_1(2)})\sigma(C_{\Delta_2(2)}))=0.
$$
Both cones $C_{\Delta_1(1)}+C_{\Delta_2(1)}$ and $C_{\Delta_1(2)}+C_{\Delta_2(2)}$ are simple (simplicial unimodular). The integer-point transform of such a cone may be expressed as the product of sums of infinite geometric series with common ratios equal to exponentials of the edges' direction vectors. 
Moreover, for almost all (all other than three) edges of the cone $C_{\Delta_1(1)}+C_{\Delta_2(1)}$ the following holds. If the direction vector of that edge is $\xi(1)$, then $C_{\Delta_1(2)}+C_{\Delta_2(2)}$ contains an edge with direction vector $\xi(2)$ such that $F(\mathbf t^{\xi(1)})=F(\mathbf t^{\xi(2)})$. This means that $1/(1-F(\mathbf t^{\xi(1)}))$ may be factored out in the course of our computation. These two observations let one carry out the computation.

The base of our induction corresponds to $\Delta$ consisting of four nodes and is covered by this case (and does not invoke the induction hypothesis).
\end{proof}

We move on to the discussion of simplicial vertices of $GT_\lambda$. For such a vertex $v$ the graph $\Gamma_v$ is made up of $n$ components each of which is a linear graph. Since every row of GT-pattern $v$ contains one element fewer than the row above, we see that all of these $n$ linear graphs start in row 0 and there is exactly one linear graph ending in each of our $n$ rows. Therefore, the multiset of values occurring in row $i+1$ of GT-pattern $v$ is obtained from the multiset of values occurring in row $i$ by the deletion of one element. This element is a coordinate of $\lambda$, i.e. $\lambda_{w_v^{-1}(i+1)}$ for some $w_v\in S_n$ (we add 1 in order to let the subscript lie within $[1,n]$). This defines the correspondence between simplicial vertices and permutations.

To complete the proof of Theorem~\ref{main} we are to establish the following.

\begin{proposition}\label{simpcontrib}
Every simplicial vertex $v$ of $GT_\lambda$ satisfies $\mu_v=w_v(\lambda)$ and
$$
F(\sigma(C_v))=w_v\bigg(\frac{e^\lambda}{\prod_{i<j}(1-x_j/x_i)}\bigg).
$$
\end{proposition}

\begin{proof}
The first equality is a direct consequence of the definitions of permutation $w_v$ and the weight $\mu_v$.

Next, note that the cone $C_v$ is simple, i.e. to compute $F(\sigma(C_v))$ it suffices to find the direction vectors of its edges. By applying Proposition~\ref{edges} to the components of $\Gamma_v$ we obtain the below description of the direction vectors.

For every pair $1\le a\le b\le n-1$ there exists exactly one edge such that its direction vector $\varepsilon$ has one nonzero coordinate in each of the rows $a,\dots,b$ and has no other nonzero coordinates. The nonzero coordinates are equal to 1 if $w_v^{-1}(a)<w_v^{-1}(b+1)$ and are equal to $-1$ otherwise. In the former case we have $F(\mathbf t^\varepsilon)=x_{b+1}/x_a$ and in the latter we have $F(\mathbf t^\varepsilon)=x_a/x_{b+1}$. In both cases the monomial on the right is equal to
$$
w_v(x_{\max(w_v^{-1}(a),w_v^{-1}(b+1))}/x_{\min(w_v^{-1}(a),w_v^{-1}(b+1))}),
$$
and the proposition follows.
\end{proof}

\section{The case of a singular weight}

In the case of a singular $\lambda$ Theorem~\ref{main} does not hold. On one hand, in this case $GT_\lambda$ does have non-simplicial vertices with nonzero contributions. On the other, the number of vertices with nonzero contributions is less than $n!$. We fix a singular integral dominant weight $\lambda$ and state our result for this case.

\begin{theorem}\label{sing}
Vertices $v$ of $GT_\lambda$ with $F(\sigma(C_v))\neq 0$ are enumerated by elements of the orbit $S_n\lambda$. Let $v_\mu$ be the vertex corresponding to $\mu\in S_n\lambda$, then 
$$
F(\sigma(C_{v_\mu}))=\sum_{w\lambda=\mu} w\bigg(\frac{e^\lambda}{\prod_{i<j}(1-x_j/x_i)}\bigg).
$$
\end{theorem}

Note that this theorem is also true in the case of a regular $\lambda$ and in that case reduces to Theorem~\ref{main}. However, our proof of Thorem~\ref{sing} makes use of Theorem~\ref{main} and, on top of that, of several notions from polyhedral combinatorics. In particular, we invoke the notion of polar duality~\cite[ch. 5]{barv} and of a normal fan of a polytope (see, for instance,~\cite{fult} or almost any other textbook on toric geometry).

Let us proceed to the proof. We first point out that Proposition~\ref{vert} holds in this case verbatim. Next, for a vertex $v$ of $GT_\la$ we may define the graph $\Gamma_v\subset T$ in complete analogy with the regular case.

Now choose an arbitrary regular integral dominant weight $\lambda'$. Then the following holds.

\begin{proposition}\label{refine}
The normal fan of $GT_{\la'}$ refines the normal fan of $GT_{\lambda}$.
\end{proposition}

\begin{proof}
It suffices to show that for any vertex $v'$ of $GT_{\lambda'}$ there exists a vertex $v$ of $GT_\la$ with 
$$
C_v-v\subset C_{v'}-v'
$$
(an inclusion of tangent cones with their vertices translated to the origin). This inclusion would follow from the inclusion $\Gamma_{v'}\subset \Gamma_v$.

As per Proposition~\ref{vert}, every coordinate of $v'$ is equal to one of the coordinates $\lambda'_j$. We define $v$ as follows: for all $j$ and all coordinates of $v'$ equal to $\la'_j$ we set the corresponding coordinate of $v$ equal to $\la_j$. On one hand, in view of Proposition~\ref{vert} with respect to $\lambda$, the obtained point is a vertex of $GT_\la$. On the other, the inclusion $\Gamma_{v'}\subset \Gamma_v$ is obvious.
\end{proof}

We have obtained a surjection
$$
\pi\colon\{\text{vertices of }GT_{\lambda'}\}\to\{\text{vertices of }GT_\lambda\},
$$
mapping $v'$ to $v$ from the above proof (i.e. such that the normal fan cone corresponding to $v$ contains the normal fan cone corresponding to $v'$).

Now we introduce several notions from the book~\cite{barv}. For an arbitrary closed convex rational (not necessarily bounded) polyhedron $P$ in our space $\mathbb R^{{n+1}\choose 2}$ let $[P]$ be its characteristic function equal to $1$ in points of $P$ and $0$ outside of $P$. Let $\mathcal P$ be the real vector space spanned by all such characteristic functions. We will refer to linear maps defined in points of $\mathcal P$ as \textit{valuations}.

\begin{theorem}[{\cite[Theorem 5.3]{barv}}]
Let $P^\circ$ be the polar dual of $P$. The map $[P]\mapsto [P^\circ]$ from $\mathcal P$ to itself extends to a certain valuation~$\mathcal D$.
\end{theorem}

Next, let $\mathcal Q\subset\mathcal P$ be the subspace spanned by functions $[P]$ with $P$ containing an affine line. For $X,Y\in\mathcal P$ write $X\approx Y$ if $X-Y\in\mathcal Q$. A key role in our proof of Theorem~\ref{sing} is played by the following statement.

\begin{lemma}\label{degen}
For a vertex $v$ of $GT_\la$ we have
$$
[C_v]\approx \sum_{\pi(v')=v} [C_{v'}+(v-v')]
$$
(on the right every cone has its vertex translated to $v$).
\end{lemma}

\begin{proof}
Let $u$ be a vertex of $GT_\la$ or $GT_{\la'}$. Let $D_u$ denote the corresponding cone in the corresponding normal fan. As per the definition of $\pi$,
\begin{equation}\label{fan}
[D_v]=\sum_{\pi(v')=v}[D_{v'}]+S;
\end{equation}
here $S$ is a linear combination of functions of the form $[K]$ where $K$ is a cone of dimension less than ${n+1}\choose 2$. For such a $K$ the cone $K^\circ$ is a cone containing an affine line. Consequently, by applying the valuation $\mathcal D$ to~(\ref{fan}) we obtain the statement of the lemma (up to a translation by $-v$).
\end{proof}

To complete the proof we will also require the following fact.

\begin{theorem}
The map $[P]\mapsto\sigma(P)$ from $\mathcal P$ to $\mathbb R(\{t_{i,j}\})$ extends to a certain valuation $\mathcal F$. Furthermore, we have $\mathcal F([P])=0$ for any function $[P]\approx 0$.
\end{theorem}

\begin{proof}
The first part follows immediately from the definition of $\sigma$. To prove the second part note that, since $P$ is rational, it contains an affine line with an integer direction vector, i.e. there exists a nonzero integer vector $u$ such that $\mathbf t^u\sigma(P)=\sigma(P)$.
\end{proof}

\begin{proof}{Proof of Theorem~\ref{sing}}
Choose a vertex $v$ of $GT_\la$ and apply valuation $\mathcal F$ to the identity from Lemma~\ref{degen} to obtain
\begin{equation}\label{degbri}
\sigma(C_v)=\sum_{\pi(v')=v}\sigma(C_{v'}+(v-v'))=\sum_{\pi(v')=v}\mathbf t^{v-v'}\sigma(C_{v'}).
\end{equation}

Next, consider a simplicial vertex $v'$ of $GT_{\la'}$ corresponding to $w\in S_n$, i.e. with $\mu_{v'}=w\la'$. The definition of $\pi$ coupled with linearity properties implies that $\mu_{\pi(v')}=w\lambda$. Consequently, the set of simplicial vertices $v'$ (i.e. such that $F(\sigma(C_{v'}))\neq 0$) in $\pi^{-1}(v)$ is

(a) empty, if $\mu_v$ is not of the form $w\la$,

(b) is comprised of those $v'$ that correspond to permutations $w'$ with $w'\lambda=w\lambda$ if $\mu_v=w\lambda$.

Theorem~\ref{sing} now follows from identity~(\ref{degbri}) and Proposition~\ref{simpcontrib}.
\end{proof}

\vskip10pt


\begin{thebibliography}{6}

\bibitem{bri} M. Brion {\it Points entiers dans les poly\`edres convexes}, Ann. Sci. \'Ecole Norm. Sup., 21:4 (1988), 653--663

\bibitem{kp}
A.~V.~Pukhlikov, A.~G.~Khovanskii,, {\it Finitely additive measures of virtual polyhedra}, Algebra i Analiz, 4:2 (1992), 161--185

\bibitem{barv}
A. I. Barvinok, {\it Integer Points in Polyhedra}, European Math. Society (EMS), Z\"urich, 2008

\bibitem{beckrob}
M. Beck, S. Robins, {\it Computing the Continuous Discretely}, Springer-Verlag, New York, 2007

\bibitem{kst}
V.~A.~Kirichenko, E.~Yu.~Smirnov, V.~A.~Timorin, {\it Schubert calculus and Gelfand--Zetlin polytopes}, Uspekhi Mat. Nauk, 67:4 (2012), 89--128

\bibitem{fult}
W. Fulton, {\it Introduction to Toric Varieties}, Princeton University Press, Princeton, NY, 1993

\end{thebibliography}
\end{document}